\documentclass[11pt]{amsart}
\usepackage[T1]{fontenc}
\usepackage{lmodern}
\usepackage[letterpaper,margin=1.1in]{geometry}
\usepackage{amsmath,amssymb,mathtools}
\usepackage{microtype}
\usepackage{xcolor}
\usepackage{hyperref}
\usepackage{xurl}
\hypersetup{
  colorlinks=true,
  linkcolor=blue!50!black,
  citecolor=blue!50!black,
  urlcolor=blue!50!black,
  pdftitle={Affine Coordinate Subspaces and Strong Khintchine Convergence},
  pdfauthor={Zirui Jiang}
}

\newtheorem{theorem}{Theorem}[section]
\newtheorem{corollary}[theorem]{Corollary}
\newtheorem{lemma}[theorem]{Lemma}
\newtheorem{proposition}[theorem]{Proposition}
\theoremstyle{remark}
\newtheorem{remark}[theorem]{Remark}

\newcommand{\R}{\mathbb R}
\newcommand{\Z}{\mathbb Z}
\newcommand{\N}{\mathbb N}
\newcommand{\T}{\mathbb T}
\newcommand{\ind}{\mathbf 1}
\newcommand{\dd}{\,\mathrm d}
\newcommand{\norm}[1]{\left\lVert #1\right\rVert}

\allowdisplaybreaks[1]
\title[Affine coordinate subspaces]{Affine coordinate subspaces and strong
Khintchine convergence}
\author{Zirui Jiang}
\address{Department of Mathematics and Statistics, University of Nevada, Reno,
Reno, Nevada 89557, USA}
\date{}

\begin{document}
\begin{abstract}
We prove that every affine coordinate subspace of positive dimension and
positive codimension fails to be of strong Khintchine type for convergence.
Combining this with Huang's criterion for Khintchine-type affine subspaces
gives a sufficient Diophantine condition for an affine coordinate subspace
to be of Khintchine type but not of strong Khintchine type for convergence.
\end{abstract}
\maketitle

\section{Introduction and main results}

Throughout, $\N=\{1,2,\ldots\}$. For a vector $y\in\R^r$, write
\[
  \norm{y}:=\inf_{p\in\Z^r}|y-p|_\infty.
\]
We use the same notation for the corresponding maximum norm on $\T^r$,
where $\T=\R/\Z$. Given a function $\psi:\N\to\R_{\ge0}$, put
\[
  S_n(\psi):=
  \{y\in\R^n:\norm{qy}<\psi(q)\text{ for infinitely many }q\in\N\}.
\]
An affine subspace $L\subset\R^n$ is of \emph{Khintchine type for convergence}
if
\[
  \sum_{q=1}^{\infty}\psi(q)^n<\infty
  \quad\Longrightarrow\quad
  |S_n(\psi)\cap L|_L=0
\]
for every non-increasing $\psi$, where $|\cdot|_L$ denotes induced Lebesgue
measure on $L$. It is of \emph{strong Khintchine type for convergence} if the
same implication holds without the monotonicity assumption. We say that $L$
is of \emph{Khintchine type} if it has the convergence property and, for every
non-increasing $\psi$ with $\sum_q\psi(q)^n=\infty$, the set
$S_n(\psi)\cap L$ has full induced measure in $L$.

Huang~\cite{huang} obtained an essentially complete classification of
Khintchine-type affine subspaces and proved a sufficient condition for the
strong convergence property. In Remark~2.6 of that paper, affine coordinate
subspaces were suggested as a natural test case, and Problem~5.2 asked for
an explicit affine subspace that is of Khintchine type for convergence but
not of strong Khintchine type for convergence. 

The analogous question for divergence is whether the full-measure
conclusion holds without the monotonicity assumption.
Ram\'irez~\cite[\S2.2]{ramirez} observed that translated coordinate
hyperplanes with an irrational fixed coordinate are not of
strong Khintchine type for divergence.

Returning to the convergence case, we follow Huang's suggestion
and prove the following theorem for affine coordinate subspaces.

\begin{theorem}\label{thm:main}
Let $d,m\ge1$, put $n=d+m$, and fix $\alpha\in\R^m$. Then the affine
coordinate subspace
\[
  L_\alpha:=\{(x,\alpha):x\in\R^d\}\subset\R^n
\]
is not of strong Khintchine type for convergence. More precisely, there
exists a function $\psi:\N\to\R_{>0}$, with $\psi(q)\to0$ and
\[
  \sum_{q=1}^{\infty}\psi(q)^n<\infty,
\]
such that $S_n(\psi)\cap L_\alpha$ has full induced Lebesgue measure in
$L_\alpha$.
\end{theorem}

Up to a permutation of coordinates, every affine coordinate subspace of
positive dimension and positive codimension has the form appearing in
Theorem~\ref{thm:main}.

We next connect this result with Huang's classification. For
$A\in\R^{(d+1)\times m}$, define
\[
 \omega(A):=\sup\left\{w>0:
 \begin{array}{l}
 \norm{Aq^T}<|q|_\infty^{-w} \;
 \text{for infinitely many }q\in\Z^m\setminus\{0\}
 \end{array}\right\},
\]
with the supremum taken in $[0,\infty]$. The coordinate subspace $L_\alpha$ is parametrized in the
notation of~\cite{huang} by
\[
  A_\alpha=\begin{pmatrix}\alpha\\0_{d\times m}\end{pmatrix}.
\]
For every $q\in\Z^m$, its zero lower rows give
$\norm{A_\alpha q^T}=\norm{\alpha q^T}$.

\begin{corollary}\label{cor:criterion}
If $\omega(A_\alpha)<d+m$, then $L_\alpha$ is of Khintchine type, hence
of Khintchine type for convergence, but is not of strong Khintchine type
for convergence.
\end{corollary}

\begin{proof}
Huang's classification theorem~\cite[Theorem~2.1]{huang} shows that
$\omega(A_\alpha)<d+m$ implies that $L_\alpha$ is of Khintchine type.
The strict inequality excludes the critical exponent appearing in the
hypotheses of that theorem. Failure of strong convergence follows from
Theorem~\ref{thm:main}.
\end{proof}

\begin{remark}
The condition $\omega(A_\alpha)<d+m$ in Corollary~1.2 holds
for Lebesgue-almost every $\alpha\in\mathbb R^m$, as follows
from the dimension formula of Bovey and Dodson
\cite[p.~338, (2)]{BoveyDodson1986}.\\
Moreover, Schmidt
\cite[p.~189]{Schmidt1970}
provides a class of algebraic parameters satisfying
this condition: take $\alpha=(\alpha_1,\ldots,\alpha_m)$
with real algebraic coordinates such that
$1,\alpha_1,\ldots,\alpha_m$ are linearly independent
over $\mathbb Q$.
For this choice, Schmidt's result, together with
Dirichlet's theorem, gives
\[
  \omega(A_\alpha)=m<d+m.
\]
\end{remark}

Section~\ref{sec:sparse} gives all required results before the proof of main theorem.
The proof of Theorem~\ref{thm:main} is given in Section~\ref{sec:main}. Related results for restricted denominators are known; for example,
Gallagher~\cite{gallagher} in dimension at least two and
Pollington--Velani--Zafeiropoulos--Zorin~\cite{pvzz} for lacunary sequences
and measures with logarithmic Fourier decay.

\section{Some Lemmata}\label{sec:sparse}

Fix an integer $d\ge1$, and let $\lambda_d$ denote normalized Lebesgue
measure on $\T^d$. For $q\in\N$ and $0<r<1/2$, define
\[
  E_d(q,r)=\{x\in\T^d:\norm{qx}<r\}.
\]
Multiplication by $q$ preserves $\lambda_d$, and the maximum-norm ball
of radius $r$ is a product of $d$ arcs of length $2r$. Hence
\[
  \lambda_d(E_d(q,r))=(2r)^d.
\]

The following lemma is a standard consequence of the overlap
estimate in \cite[\S3.2]{Harman1998}.
We include an elementary proof with an explicit constant for completeness.

\begin{lemma}\label{lem:correlation}
Let $q>s\ge1$ and $0<r,u<1/2$. Then
\[
 \left|\lambda_d(E_d(q,r)\cap E_d(s,u))
 -\lambda_d(E_d(q,r))\lambda_d(E_d(s,u))\right|
 \le \frac{4ds}{q}.
\]
\end{lemma}

\begin{proof}
We first work on $\T$. For a positive integer $a$ and any real $v$, set
\[
  h_{a,v}(x)=\ind_{\{\norm{ax}<v\}}.
\]
Multiplication by $a$ preserves Haar measure, so the mass of this event
is independent of $a$ and equals
\[
  M(v):=\int_\T h_{a,v}(x)\dd\lambda_1(x)
       =\max\{\min\{1,2v\},0\}.
\]
This formula also handles nonpositive radii and radii at least $1/2$.
At $v=1/2$, the omitted boundary set has measure zero.

Put
\[
  J=\int_\T h_{q,r}(x)h_{s,u}(x)\dd\lambda_1(x),
  \qquad \delta=\frac{s}{q}.
\]
For a fixed $x\in\T$, the identity $q(x+t/q)=qx+t$ on $\T$ gives
\[
  \int_0^1 h_{q,r}(x+t/q)\dd t=2r.
\]
All functions used below are bounded and measurable on finite measure
spaces. Fubini's theorem therefore implies, for every real $v$,
\[
  \int_0^1\int_\T
  h_{q,r}(x+t/q)h_{s,v}(x)\dd\lambda_1(x)\dd t
  =2r\,M(v).
\]

For $0\le t\le1$, write $z=t/q\bmod\Z$. Then $\norm{sz}\le s/q=\delta$.
The triangle inequality in both directions gives the pointwise bounds
\[
  h_{s,u-\delta}(x)\le h_{s,u}(x+z)
  \le h_{s,u+\delta}(x).
\]
Multiplying by $h_{q,r}(x+z)$, integrating in $x$, and using translation
invariance for the middle integral, we get
\begin{align*}
  \int_\T h_{q,r}(x+z)h_{s,u-\delta}(x)\dd\lambda_1(x)
  &\le J
  \le
  \int_\T h_{q,r}(x+z)h_{s,u+\delta}(x)\dd\lambda_1(x).
\end{align*}
Averaging over $t\in(0,1]$ yields
\[
  2r\,M(u-\delta)\le J\le2r\,M(u+\delta).
\]
Since $0<u<1/2$ and $\delta\ge0$, the clipped mass formula gives
\[
  2u-2\delta\le M(u-\delta),
  \qquad M(u+\delta)\le2u+2\delta.
\]
Thus
\[
  4ru-4r\delta\le J\le4ru+4r\delta,
  \qquad
  |J-4ru|\le4r\delta\le2\delta\le\frac{4s}{q}.
\]

The coordinate conditions form product sets. Consequently,
\[
  \lambda_d(E_d(q,r)\cap E_d(s,u))=J^d,\qquad
  \lambda_d(E_d(q,r))\lambda_d(E_d(s,u))=(4ru)^d.
\]
For $A,B\in[0,1]$, the standard power-difference estimate gives
\[
  |A^d-B^d|
  \le d|A-B|\max\{|A|,|B|\}^{d-1}
  \le d|A-B|,
\]
where the last inequality uses $d\ge1$ and $\max\{|A|,|B|\}\le1$.
Apply this with $A=J$ and $B=4ru$ to complete the proof.
\end{proof}

\begin{proposition}\label{prop:sparse}
Let $(q_k)_{k\ge1}$ be a strictly increasing sequence of positive integers,
and let $(r_k)_{k\ge1}\subset(0,1/2)$. Suppose that
\begin{equation}\label{eq:divergence}
  \sum_{k=1}^{\infty}r_k^d=\infty
\end{equation}
and
\begin{equation}\label{eq:sparsity}
  \sum_{k=2}^{\infty}\frac{q_1+\cdots+q_{k-1}}{q_k}<\infty.
\end{equation}
Then, for $\lambda_d$-almost every $x\in\T^d$,
\[
  \norm{q_kx}<r_k\qquad\text{for infinitely many }k.
\]
\end{proposition}

\begin{proof}
Set $E_k=E_d(q_k,r_k)$, $p_k=\lambda_d(E_k)$, and
\[
  S_N(x)=\sum_{k=1}^N\ind_{E_k}(x),
  \qquad B_N=\sum_{k=1}^Np_k=2^d\sum_{k=1}^Nr_k^d.
\]
By~\eqref{eq:divergence}, $B_N\to\infty$.
The finite-sum identity
\[
  \left(\sum_{k=1}^N a_k\right)^2
  =\sum_{k=1}^N a_k^2+2\sum_{1\le i<k\le N}a_i a_k
\]
follows by induction on $N$. Applying it to $a_k=\ind_{E_k}(x)$,
using $\ind_{E_k}^2=\ind_{E_k}$, and integrating gives
\[
  \int_{\T^d}S_N^2\dd\lambda_d
  =\sum_{k=1}^Np_k+
    2\sum_{1\le i<k\le N}\lambda_d(E_i\cap E_k).
\]
Subtracting $B_N^2$, we obtain the variance identity
\[
  \int_{\T^d}|S_N-B_N|^2\dd\lambda_d
  =\sum_{k=1}^N(p_k-p_k^2)
  +2\sum_{1\le i<k\le N}
    \bigl(\lambda_d(E_i\cap E_k)-p_ip_k\bigr).
\]
Since $p_k-p_k^2\le p_k$, Lemma~\ref{lem:correlation} implies
\begin{align*}
  \int_{\T^d}|S_N-B_N|^2\dd\lambda_d
  &\le B_N+8d\sum_{k=2}^N\frac{q_1+\cdots+q_{k-1}}{q_k}
  \le B_N+C,
\end{align*}
where
\[
  C=8d\sum_{k=2}^{\infty}\frac{q_1+\cdots+q_{k-1}}{q_k}<\infty.
\]
No independence of the events $E_k$ is assumed.

For every $\ell\ge1$, choose any integer $N_\ell$ with
$B_{N_\ell}\ge\ell^2$, and define
\[
  D_\ell=
  \{x:|S_{N_\ell}(x)-B_{N_\ell}|\ge B_{N_\ell}/2\}.
\]
Chebyshev's inequality gives
\[
  \lambda_d(D_\ell)
  \le\frac{B_{N_\ell}+C}{(B_{N_\ell}/2)^2}
  =\frac4{B_{N_\ell}}+\frac{4C}{B_{N_\ell}^2}
  \le\frac4{\ell^2}+\frac{4C}{\ell^4}.
\]
The right-hand side is summable. By the first Borel--Cantelli lemma,
for almost every $x$ and all sufficiently large $\ell$,
\[
  S_{N_\ell}(x)>B_{N_\ell}/2\ge\ell^2/2.
\]
Hence the sequence $S_N(x)$ is unbounded. Since it is non-decreasing in
$N$, it tends to infinity. Thus almost every $x$ belongs to infinitely
many of the sets $E_k$.
\end{proof}

\begin{lemma}\label{lem:returns}
Let $\alpha\in\R^m$. For every integer $M\ge1$ and every $\eta>0$, there
exists an integer $q>M$ such that $\norm{q\alpha}<\eta$.
\end{lemma}

\begin{proof}
By Dirichlet's theorem, there are
infinitely many positive integers $q$ such that
\[
  \norm{q\alpha}\le q^{-1/m}.
\]
These integers are unbounded, so we may choose one satisfying
\[
  q>\max\{M,\eta^{-m}\}.
\]
Then $q>M$ and
\[
  \norm{q\alpha}\le q^{-1/m}<\eta,
\]
as required.
\end{proof}

\section{Proof of the main theorem}\label{sec:main}

\begin{proof}[Proof of Theorem~\ref{thm:main}]
Choose a real number $a$ with
\[
  \frac1{d+m}<a<\frac1d,
\]
and put $r_k=\frac18k^{-a}$ for $k\ge1$.
In particular, $0<r_k\le1/8<1/2$.
Applying Lemma~\ref{lem:returns} recursively, choose $q_1>1$ and a
sequence of positive integers $q_k$ such that
\begin{equation}\label{eq:recursive}
 \begin{gathered}
  \norm{q_k\alpha}<r_k\qquad(k\ge1),\\
  q_k>2^k(q_1+\cdots+q_{k-1})\qquad(k\ge2).
 \end{gathered}
\end{equation}
This sequence is strictly increasing, and its prefix ratios satisfy
\[
  0\le\frac{q_1+\cdots+q_{k-1}}{q_k}
  \le2^{-k}\qquad(k\ge2).
\]
Clearly it satisfies
\eqref{eq:sparsity}. Since $ad<1$, the $p$-series test gives
$\sum_kr_k^d=\infty$. Proposition~\ref{prop:sparse} therefore implies
\begin{equation}\label{eq:freehits}
  \norm{q_kx}<r_k\qquad\text{for infinitely many }k
\end{equation}
for $\lambda_d$-almost every $x\in\T^d$.

Define
\[
  \psi(q)=
  \begin{cases}
    \dfrac14 k^{-a},&q=q_k\text{ for some }k\ge1,\\[4pt]
    \dfrac1{4(q+1)^2},&q\notin\{q_k:k\ge1\}.
  \end{cases}
\]
The strict increase of $q_k$ makes this definition unambiguous, and
$\psi(q)>0$ for all $q$. With $n=d+m$, we have $an>1$.
Summing the prescribed values along the injective map $k\mapsto q_k$
and bounding the remaining terms by the background series gives
\[
  \sum_{q=1}^{\infty}\psi(q)^n
  \le4^{-n}\left(
      \sum_{k=1}^{\infty}k^{-an}
      +\sum_{q=1}^{\infty}(q+1)^{-2n}
    \right)<\infty.
\]

For every $k$ occurring in~\eqref{eq:freehits}, condition~\eqref{eq:recursive}
gives
\[
  \norm{q_k(x,\alpha)}
  =\max\{\norm{q_kx},\norm{q_k\alpha}\}
  <r_k<2r_k=\psi(q_k).
\]
The denominators $q_k$ are distinct, so this yields infinitely many
simultaneous approximations for almost every free torus coordinate.

Let $\pi:\R^d\to\T^d$ be the coordinatewise quotient map,
$C_0=(0,1]^d$, and $\mathcal L^d$ Lebesgue measure on $\R^d$. Set
\[
  G=\bigcap_{K=1}^{\infty}\bigcup_{k\ge K}E_d(q_k,r_k),
  \qquad N_0=C_0\cap\pi^{-1}(\T^d\setminus G).
\]
By~\eqref{eq:freehits} and measure preservation of $\pi$ on $C_0$,
\[
  \mathcal L^d(N_0)=\lambda_d(\T^d\setminus G)=0.
\]
For $x\in\R^d$, $z\in\Z^d$, and $k\ge1$, we have $q_kz\in\Z^d$ and hence
\[
  \norm{q_k(x+z)}
  =\inf_{p\in\Z^d}|q_kx+q_kz-p|_\infty
  =\norm{q_kx}.
\]
Since $\R^d=\bigsqcup_{z\in\Z^d}(C_0+z)$ and $\Z^d$ is countable,
translation invariance of $\mathcal L^d$ gives
\begin{gather*}
  N:=\pi^{-1}(\T^d\setminus G)
    =\bigcup_{z\in\Z^d}(N_0+z),\\
  \mathcal L^d(N)
  \le\sum_{z\in\Z^d}\mathcal L^d(N_0+z)
  =\sum_{z\in\Z^d}\mathcal L^d(N_0)=0.
\end{gather*}
Writing $\iota_\alpha(x)=(x,\alpha)$, the preceding approximation gives
$L_\alpha\setminus S_n(\psi)\subseteq\iota_\alpha(N)$. Since
$\iota_\alpha$ is an isometric parametrization,
\[
  |L_\alpha\setminus S_n(\psi)|_{L_\alpha}
  \le |\iota_\alpha(N)|_{L_\alpha}
  =\mathcal L^d(N)=0.\qedhere
\]
\end{proof}

\end{document}